\documentclass[11pt]{amsart}
\usepackage{geometry, graphicx,amssymb,esint, verbatim}
\usepackage[colorlinks=true,urlcolor=blue, citecolor=red,linkcolor=blue,
linktocpage,pdfpagelabels, bookmarksnumbered,bookmarksopen]{hyperref}
\usepackage[hyperpageref]{backref}

\usepackage{color}

\newtheorem{lm}{Lemma}[section]
\newtheorem{prop}[lm]{Proposition}
\newtheorem{teo}[lm]{Theorem}
\newtheorem{coro}[lm]{Corollary}
\theoremstyle{definition}
\newtheorem{oss}[lm]{Remark}
\newtheorem{defi}[lm]{Definition}
\newtheorem*{ack}{Acknowledgments}

\author[Brasco]{L. Brasco}
\author[Cinti]{E. Cinti}

\address[L.\ Brasco]{Dipartimento di Matematica e Informatica
	\newline\indent
	Universit\`a degli Studi di Ferrara
	\newline\indent
	Via Machiavelli 35, 44121 Ferrara, Italy}
\address{{\it and }
	Aix-Marseille Universit\'e, CNRS
	\newline\indent
	Centrale Marseille, I2M, UMR 7373, 39 Rue Fr\'ed\'eric Joliot Curie
	\newline\indent
	13453 Marseille, France}
\email{c}

\address[E.\ Cinti]{Dipartimento di Matematica
	\newline\indent
	Universit\`a degli Studi di Bologna
	\newline\indent
	 Piazza di Porta San Donato 5, 40126 Bologna, Italy}
\email{eleonora.cinti5@unibo.it}


\numberwithin{equation}{section}

\date{\today}

\title[Fractional Hardy inequality]{On fractional Hardy inequalities\\in convex sets}

\begin{document}

\begin{abstract}
We prove a Hardy inequality on convex sets, for fractional Sobolev-Slobodecki\u{\i} spaces of order $(s,p)$. The proof is based on the fact that in a convex set the distance from the boundary is a superharmonic function, in a suitable sense. The result holds for every $1<p<\infty$ and $0<s<1$, with a constant which is stable as $s$ goes to $1$.
\end{abstract}

\maketitle
\begin{center}
\begin{minipage}{8cm}
\small
\tableofcontents
\end{minipage}
\end{center}

\section{Introduction}

\subsection{A quick overview on Hardy inequality} Given an open set $\Omega \subset \mathbb{R}^N$ with Lipschitz boundary, we will use the notation
\[
d_\Omega(x):=
\left\{\begin{array}{rl}
\inf\limits_{y\in \partial \Omega} |x-y|, & \mbox{ if } x\in \Omega,\\
0, & \mbox{ if } x\in \mathbb{R}^N\setminus \Omega.
\end{array}
\right.
\]
A fundamental result in the theory of Sobolev spaces is the {\it Hardy inequality}
\begin{equation}
\label{hardy2}
C_\Omega\,\int_\Omega \frac{|u|^2}{d_\Omega^2}\,dx\le \int_\Omega |\nabla u|^2\,dx,\qquad \mbox{ for every }u\in C^\infty_0(\Omega),
\end{equation}
see for example \cite[Theorem 21.3]{OK}.
It is well-known that for a convex set $K\subset \mathbb R^N$, such an inequality holds with the dimension-free universal constant 
\[
C_K=\frac{1}{4},
\]
see for example \cite[Theorem 2]{Da}.
Moreover, such a constant is sharp. In order to explain the aims and techniques of the present paper, it is useful to recall a proof of this fact. 
\par
A well-known and very elegant way of proving \eqref{hardy2} with sharp constant for convex sets, consists in mimicking {\it Moser's logarithmic estimate} for positive supersolutions of elliptic partial differential equations. 
The starting point is the observation that on a convex set $K$ we have $-\Delta d_K\ge 0$, i.e. the distance function is superharmonic. More precisely, it holds
\begin{equation}
\label{eq}
\int_K \langle \nabla d_K,\nabla \varphi\rangle\,dx\ge 0,\qquad \mbox{ for every nonnegative }\varphi\in C^\infty_0(K).
\end{equation}
By following Moser (see \cite[page 586]{Mo}), one can test the equation \eqref{eq} with\footnote{Of course, such a test function is not in $C^\infty_0(K)$. However, by a standard density argument, in \eqref{eq} we can allow $W^{1,2}$ test functions with compact support. We avoid this unessential technicality for ease of readability.} 
\[
\varphi=\frac{u^2}{d_K},
\]
where $u\in C^\infty_0(K)$. This gives
\[
2\,\int_K \left\langle \frac{\nabla d_K}{d_K},\nabla u\right\rangle\,u\,dx-\int_K \left|\frac{\nabla d_K}{d_K}\right|^2\,u^2\,dx\ge 0,
\]
that is
\[
\int_K \left|\frac{\nabla d_K}{d_K}\right|^2\,u^2\,dx\le 2\,\int_K \left\langle \frac{\nabla d_K}{d_K},\nabla u\right\rangle\,u\,dx.
\]
We now use Young's inequality 
\[
\langle a,b\rangle\le \frac{|a|^2}{2}+\frac{|b|^2}{2},\qquad a,b\in\mathbb{R}^N,
\]
with the following choices
\[
a=\sqrt{\delta}\,\frac{\nabla d_K}{d_K}\,u\qquad \mbox{ and }\qquad b=\frac{1}{\sqrt{\delta}}\,\nabla u,
\]
where $\delta$ is an arbitrary positive real number.
This leads to
\[
\int_K \left|\frac{\nabla d_K}{d_K}\right|^2\,u^2\,dx\le \delta\,\int_K \left|\frac{\nabla d_K}{d_K}\right|^2\,u^2\,dx+\frac{1}{\delta}\,\int_K |\nabla u|^2\,dx,
\]
which can be recast as 
\[
\delta\,(1-\delta)\,\int_K \left|\frac{\nabla d_K}{d_K}\right|^2\,u^2\,dx\le \int_K |\nabla u|^2\,dx.
\]
It is now sufficient to observe that the term $\delta\,(1-\delta)$ in the left-hand side is maximal for $\delta=1/2$. This leads to the Hardy inequality with the claimed sharp constant $1/4$, once it is observed that
\[
|\nabla d_K|=1,\qquad \mbox{ a.\,e. in }K.
\]
The latter implies that more generally for every $1<p<\infty$ we have $-\Delta_p \,d_K\ge 0$ in $K$, i.e. $d_K$ is $p-$superharmonic in the following sense
\[
\int_K \langle |\nabla d_K|^{p-2}\,\nabla d_K,\nabla \varphi\rangle\,dx\ge 0,\qquad \mbox{ for every nonnegative }\varphi\in C^\infty_0(K).
\]
By testing this with $\varphi=|u|^p/d_K^{p-1}$ and suitably adapting the proof above, one can prove the more general Hardy inequality for convex sets
\begin{equation}
\label{localp}
\left(\frac{p-1}{p}\right)^p\,\int_K \frac{|u|^p}{d_K^p}\,dx\le \int_K |\nabla u|^p\,dx.
\end{equation}
Once again, the constant appearing in \eqref{localp} is sharp and independent of both $K$ and the dimension $N$.

\subsection{Main result} The scope of the present paper is to prove a {\it fractional version} of Hardy inequality for convex sets, by adapting to the fractional setting the Moser-type proof presented above. An essential feature of our method is that the relevant constant appearing in the Hardy inequality {\bf is stable as the fractional order of differentiability $s$ converges to} $1$, see Remark \ref{oss:s} below.
More precisely, we prove the following
\begin{teo}[Hardy inequality on convex sets]
\label{teo:torino}
Let $1<p<\infty$ and $0<s<1$. Let $K\subset\mathbb{R}^N$ be an open convex set such that $K\not= \mathbb{R}^N$. Then for every $u\in C^\infty_0(K)$ we have
\begin{equation}
\label{nostra}
\frac{\mathcal{C}}{s\,(1-s)}\,\int_K \frac{|u|^p}{d_K^{s\,p}}\,dx\le \iint_{\mathbb{R}^N\times\mathbb{R}^N} \frac{|u(x)-u(y)|^p}{|x-y|^{N+s\,p}}\,dx\,dy,
\end{equation}
for a computable constant $\mathcal{C}=\mathcal{C}(N,p)>0$.
\end{teo}
The constant obtained in \eqref{nostra} is very likely not sharp. However, a couple of comments are in order on this point.
\begin{oss}[Asymptotic behaviour in $s$ of the constant]
\label{oss:s}
We recall that if $u\in C^\infty_0(K)$, then we have  (see \cite[Corollary 1.3]{Po} and \cite[Proposition 2.8]{BPS})
\[
\lim_{s\nearrow 1} (1-s)\,\iint_{\mathbb{R}^N\times\mathbb{R}^N} \frac{|u(x)-u(y)|^p}{|x-y|^{N+s\,p}}\,dx\,dy=\alpha_{N,p}\,\int_{K} |\nabla u|^p\,dx,
\]
with
\[
\alpha_{N,p}=\frac{1}{p}\,\int_{\mathbb{S}^{N-1}}|\langle \omega,\mathbf{e}_1\rangle|^p\,d\mathcal{H}^{N-1}(\omega),\qquad \mathbf{e}_1=(1,0,\dots,0). 
\]
In this respect, we observe that the constant appearing in Theorem \ref{teo:torino} has the correct asymptotic behaviour as $s$ converges to $1$: by passing to the limit in \eqref{nostra} as $s$ goes to $1$, we obtain the usual local Hardy inequality
\begin{equation}
\label{passa}
\mathcal{C}\,\int_K \frac{|u|^p}{d_K^p}\,dx\le \alpha_{N,p}\,\int_K |\nabla u|^p\,dx.
\end{equation}
As for  the limit $s\searrow 0$, we recall that (see \cite[Theorem 3]{MS})
\[
\lim_{s\searrow 0} s\,\iint_{\mathbb{R}^N\times\mathbb{R}^N} \frac{|u(x)-u(y)|^p}{|x-y|^{N+s\,p}}\,dx\,dy=\beta_{N,p}\,\int_{K}|u|^p\,dx,
\]
with
\[
\beta_{N,p}=\frac{2\,N\,\omega_N}{p},
\]
and $\omega_N$ is the volume of the $N-$dimensional unit ball. Thus, in this case as well, our constant in \eqref{nostra} exhibits the correct asymptotic behaviour as $s$ converges to $0$. 
\end{oss}
\begin{oss}[Dependence on $N$ of $\mathcal{C}$]
At a first glance, it may look strange that the constant $\mathcal{C}$ in Theorem \ref{teo:torino} depends on $N$. Indeed, we have seen in \eqref{localp} that in the local case such an inequality holds with the universal constant
\[
\left(\frac{p-1}{p}\right)^p.
\]
However, it is easily seen that $\mathcal{C}$ {\it must} depend on $N$ in the fractional case. Indeed, we have already seen that passing to the limit in \eqref{nostra} as $s$ goes to $1$, we obtain the local Hardy inequality \eqref{passa}. As we already said, the sharp constant in the previous inequality is $((p-1)/p)^p$, which means that we must have
\[
\mathcal{C}\le \left(\frac{p-1}{p}\right)^p\,\alpha_{N,p},\qquad \mbox{ for every }N\ge 1.
\]
On the other hand, it is easily seen that $\alpha_{N,p}$ converges to $0$ as $N$ goes to $\infty$. This shows that $\mathcal{C}$ in \eqref{nostra} must depend on $N$.
\end{oss}

\subsection{Method of proof} We now spend some words on the proof of Theorem \ref{teo:torino}. We first observe that there is an elementary proof of the inequality
\[
C_{N,p,s}\,\int_K \frac{|u|^p}{d_K^{s\,p}}\,dx\le \iint_{\mathbb{R}^N\times\mathbb{R}^N} \frac{|u(x)-u(y)|^p}{|x-y|^{N+s\,p}}\,dx\,dy,
\]
as pointed out to us by Bart\l omiej Dyda.
This is based on geometric considerations and on the nonlocality of the double integral in the right-hand side. We detail this argument in Subsection \ref{sec:rotto} below. We point out that this method is purely nonlocal and does not have a local counterpart.
\par
Then it is not surprising that with this method we obtain a constant $C_{N,p,s}>0$ such that
\[
C_{N,p,s}\sim \frac{1}{s}\qquad \mbox{ as } s \searrow 0,
\]
while 
\[
C_{N,p,s}=o\left(\frac{1}{1-s}\right)\qquad \mbox{ as } s \nearrow 1.
\] 
As explained in Remark \ref{oss:s}, this means that the constant obtained in this way does not have the correct asymptotic dependence on $s$ as this goes to $1$. This suggests that this proof is not the correct one for $s\sim 1$. 
\par
For this reason, in order to obtain a constant behaving as $1/(1-s)$, 
we use a nonlocal variant of the Moser-type proof recalled at the beginning.
This is based on the fact that for every $0<s<1$ and $1<p<\infty$ we have in weak sense
\[
(-\Delta_p)^s d_K^s\ge 0 \qquad \mbox{ in }K,
\]
where $(-\Delta_p)^s$ is the {\it fractional $p-$Laplacian of order} $s$.
In other words, 
the function $d_K^s$ {\it is $(s,p)-$superharmonic} in the following sense (see Proposition \ref{prop:guido} below)
\[
\iint_{\mathbb{R}^N\times\mathbb{R}^N} \frac{|d_K(x)^s-d_K(y)^s|^{p-2}\,(d_K(x)^s-d_K(y)^s)\,\big(\varphi(x)-\varphi(y)\big)}{|x-y|^{N+s\,p}}\,dx\,dy\ge 0,
\]
for every nonnegative and smooth function $\varphi$, with compact support in $K$. Then we will test this inequality with $\varphi=|u|^p/d_K^{s\,(p-1)}$. 
\par
As in the local case, this trick is an essential feature in order to prove BMO regularity of the logarithm of positive supersolutions to the fractional $p-$Laplacian. This in turn is a crucial step in the proof of H\"older continuity of solutions to equations involving $(-\Delta_p)^s$. In this respect, this idea has already been exploited by Di Castro, Kuusi and Palatucci in \cite[Lemma 1.3]{DKP} (see also \cite[Lemma 3.4]{Ka} for the case $p=2$). However, we observe that the computations in \cite[Lemma 1.3]{DKP} do not lead to the desired Hardy inequality, due to a lack of symmetry in $x$ and $y$. For this, we need finer algebraic manipulations and a subtler pointwise inequality: these are contained in Lemma \ref{lm:bacini}, which is one of the main ingredients of the proof of Theorem \ref{teo:torino}. We refer to Remark \ref{oss:spiegone} below for a more detailed discussion on this point.
\vskip.2cm
The quest for fractional Hardy inequalities is certainly not new.  We list below some related contributions.
\begin{oss}[Comparison with known results] 
For the case of the whole space, the following fractional Hardy inequality
\[
C_{N,s,p}\,\int_{\mathbb{R}^N} \frac{|u|^p}{|x|^{s\,p}}\,dx\le \iint_{\mathbb{R}^N\times\mathbb{R}^N} \frac{|u(x)-u(y)|^p}{|x-y|^{N+s\,p}}\,dx\,dy,\qquad s\,p\not =N,
\]
has been proved by Maz'ya and Shaposhnikova in \cite[Theorem 2]{MS} and by Frank and Seiringer in \cite[Theorem 1.1]{FSspace}. In \cite{FSspace}, the sharp value of the constant $C_{N,s,p}$ is obtained.
We also refer to \cite[Theorem 1.4]{CF}, \cite[Theorem 1]{DV} and \cite[Theorem 6.1]{HKP} for some weighted versions of this inequality.
\par
As far as subsets $\Omega\subset\mathbb{R}^N$ are concerned, we would like to mention that in \cite[Theorem 1.1]{Dy} Dyda proved
\begin{equation}
\label{localizzata}
C_\Omega\,\int_\Omega \frac{|u|^p}{d_\Omega^{s\,p}}\,dx\le \iint_{\Omega\times \Omega} \frac{|u(x)-u(y)|^p}{|x-y|^{N+s\,p}}\,dx\,dy,\qquad \mbox{ for every }u\in C^\infty_0(\Omega),\end{equation}
under suitable assumptions on the open Lipschitz set $\Omega\subset \mathbb R^N$ and some restrictions on the product $s\,p$, see also \cite[Corollary 3]{DV}. 
\par
Observe that in the right-hand side of \eqref{localizzata}, the fractional Sobolev seminorm is now computed on $\Omega\times\Omega$, rather than on the whole $\mathbb{R}^N\times\mathbb{R}^N$.
However, as pointed out in \cite{Dy}, such a stronger inequality {\it fails to hold} for $s\,p\le 1$, whenever $\Omega$ is bounded. 
\par
On the other hand, when $\Omega$ is a half-space, inequality \eqref{localizzata} holds for $s\,p\not =1$. In this case, the sharp constant has been computed by Bogdan and Dyda in \cite[Theorem 1]{BD} for $p=2$ and by Frank and Seiringer in \cite[Theorem 1.1]{FS} for a general $1<p<\infty$.
We also mention that when $s\,p>1$ and $\Omega\subset\mathbb{R}^N$ is an open convex set, inequality \eqref{localizzata} with sharp constant (which is the same as in the half-space) has been proved by Loss and Sloane in \cite[Theorem 1.2]{LS}.
\par
We point out that our proof is different from those of the aforementioned results and our Hardy inequality \eqref{nostra} holds without any restriction on the product $s\,p$.
\end{oss}
%


\subsection{Plan of the paper}
We start with Section \ref{sec:2}, containing the main notations, definitions and some technical results. In this part, the main point is Proposition \ref{prop:freddo}. In Section \ref{sec:3} we show that, in a convex set $K$, the distance function $d_K$ raised to the power $s$ is $(s,p)-$superharmonic, see Proposition \ref{prop:guido}. 
The proof of Theorem \ref{teo:torino} is contained in Section \ref{sec:4}. Finally, in Section \ref{sec:consequence} we highlight some applications of our main result. 
The paper is complemented with an Appendix, containing some pointwise inequalities which are crucially exploited in the proof of our main result.

\begin{ack}
We wish to thank Guido De Philippis for suggesting us the elegant geometric argument in the proof of Proposition \ref{prop:guido}. We also thank Tuomo Kuusi for a discussion which clarified some points of his paper \cite{DKP}. Xavier Cabr\'e, Bart\l omiej Dyda and Rupert L. Frank made some useful comments on a preliminary version of the paper, we warmly thank them.
\par
E. Cinti is supported by the MINECO grant MTM2014-52402-C3-1-P, the ERC Advanced Grant 2013
n.   339958  {\it Complex  Patterns  for  Strongly  Interacting  Dynamical  Systems  -  COMPAT} and is part of the Catalan research group 2014 SGR 1083. 
\par
Both authors are members of the Gruppo Nazionale per l'Analisi Matematica, la Probabilit\`a
e le loro Applicazioni (GNAMPA) of the Istituto Nazionale di Alta Matematica (INdAM).
\end{ack}

\section{Preliminaries}
\label{sec:2}

\subsection{Notations}
For $x_0\in\mathbb{R}^N$ and $R>0$, we use the standard notation
\[
B_R(x_0)=\{x\in\mathbb{R}^N\, :\, |x-x_0|<R\}.
\]
For notational simplicity, for every $1<p<\infty$ we introduce the function $J_p:\mathbb{R}\to\mathbb{R}$ defined by
\[
J_p(t)=|t|^{p-2}\,t,\qquad \mbox{ for }t\in\mathbb{R}.
\]
For $\alpha>0$, we also set 
\[
L^{p-1}_{\alpha}(\mathbb{R}^N)=\left\{u\in L^{p-1}_{\rm loc}(\mathbb{R}^N)\, :\, \int_{\mathbb{R}^N} \frac{|u(x)|^{p-1}}{(1+|x|)^{N+\alpha}}\,dx<+\infty\right\}.
\]
If $\Omega\subset\mathbb{R}^N$ is an open set, for every $1<p<\infty$ and $0<s<1$, we define
\[
W^{s,p}(\Omega)=\{u\in L^p(\Omega)\, :\, [u]_{W^{s,p}(\Omega)}<+\infty\},
\]
where
\[
[u]_{W^{s,p}(\Omega)}=\left(\iint_{\Omega\times\Omega} \frac{|u(x)-u(y)|^p}{|x-y|^{N+s\,p}}\,dx\,dy\right)^\frac{1}{p}.
\]
The local version $W^{s,p}_{\rm loc}(\Omega)$ is defined in the usual way.
\subsection{Functional analytic facts}
We start with the following
\begin{defi}
\label{defi:superarmonica}
Let $1<p<\infty$ and $0<s<1$. Let $\Omega \subset \mathbb R^N$ be an open set. We say that $u\in W^{s,p}_{\rm loc}(\Omega)\cap L^{p-1}_{s\,p}(\mathbb{R}^N)$ is:
\begin{itemize}
\item  {\it locally weakly $(s,p)-$superharmonic} in $\Omega$ if 
\begin{equation}\label{super}
\iint_{\mathbb{R}^N\times\mathbb{R}^N} \frac{J_p(u(x)-u(y))\,\big(\varphi(x)-\varphi(y)\big)}{|x-y|^{N+s\,p}}\,dx\,dy\ge 0,
\end{equation}
for every nonnegative $\varphi\in W^{s,p}(\Omega)$ with compact support in $\Omega$; 
\vskip.2cm
\item  {\it locally weakly $(s,p)-$subharmonic} in $\Omega$ if $-u$ is $(s,p)-$superharmonic in $\Omega$;
\vskip.2cm
\item
{\it locally weakly $(s,p)-$harmonic} in $\Omega$ if it is both $(s,p)-$superharmonic\\ and $(s,p)-$subharmonic.
\end{itemize}
\end{defi}
We observe that thanks to the assumptions on $u$, the double integral in \eqref{super} is finite for every admissible test function.
\vskip.2cm
The following simple result is quite standard. We include the proof for completeness.
\begin{lm}
\label{lm:cagatella}
Let $\Omega\subset\mathbb{R}^N$ be a bounded measurable set. Then for every $u\in L^{p-1}_{\alpha}(\mathbb{R}^N)$ and $r>0$ we have
\[
\sup_{x\in\Omega} \int_{\mathbb{R}^N\setminus B_r(x)} \frac{|u(y)|^{p-1}}{|x-y|^{N+\alpha}}\,dy<+\infty.
\]
\end{lm}
\begin{proof}
We observe that for every $x\in\Omega$ and $y\not\in B_r(x)$ we have
\[
\frac{|x-y|}{1+|y|}\ge \frac{|x-y|}{1+|x-y|+|x|}=\frac{1}{1+\dfrac{1+|x|}{|x-y|}}\ge \frac{r}{r+1+|x|}.
\]
This gives immediately
\[
\int_{\mathbb{R}^N\setminus B_r(x)} \frac{|u(y)|^{p-1}}{|x-y|^{N+s\,p}}\,dy\le \left(\frac{r+1+|x|}{r}\right)^{N+s\,p}\,\int_{\mathbb{R}^N} \frac{|u(y)|^{p-1}}{(1+|y|)^{N+s\,p}}\,dy.
\]
Since $\Omega$ is bounded, we get the desired conclusion.
\end{proof}
The following technical result will be used in the next section.
\begin{lm}
\label{lm:madonnabellina}
Let $1<p<\infty$ and $0<s<1$. Let $\Omega\subset \mathbb{R}^N$ be an open bounded set. Given $u\in W^{s,p}_{\rm loc}(\Omega)\cap L^{p-1}_{s\,p}(\mathbb{R}^N)$, $\varphi \in L^p(\mathbb{R}^N)$ with compact support in $\Omega$ and  $\varepsilon>0$, the function
\[
(x,y)\mapsto \frac{J_p(u(x)-u(y))}{|x-y|^{N+s\,p}}\,\varphi(x),
\]
is summable on $\mathcal{T}_\varepsilon:=\{(x,y)\in\mathbb{R}^N\times\mathbb{R}^N\, :\, |x-y|\ge \varepsilon\}$.
\end{lm}
\begin{proof}
Let us call $\mathcal{O}$ the support of $\varphi$. We have
\[
\begin{split}
\iint_{\mathcal{T}_\varepsilon} &\left|\frac{J_p(u(x)-u(y))}{|x-y|^{N+s\,p}}\,\varphi(x)\right|\,dx\,dy\\&=\iint_{\{(x,y)\in\mathcal{O}\times\mathbb{R}^N\, :\, |x-y|\ge \varepsilon\}}\frac{|u(x)-u(y)|^{p-1}}{|x-y|^{N+s\,p}}\,|\varphi(x)|\,dx\,dy\\
&\le \varepsilon^{-\frac{N}{p}-s}\,\iint_{\{(x,y)\in\mathcal{O}\times\mathcal{O}\, :\, |x-y|\ge \varepsilon\}}\frac{|u(x)-u(y)|^{p-1}}{|x-y|^\frac{N+s\,p}{p'}}\,|\varphi(x)|\,dx\,dy\\
&+C\,\iint_{\{(x,y)\in\mathcal{O}\times(\mathbb{R}^N\setminus\mathcal{O})\, :\, |x-y|\ge \varepsilon\}}\frac{|u(x)|^{p-1}+|u(y)|^{p-1}}{|x-y|^{N+s\,p}}\,|\varphi(x)|\,dx\,dy\\
&\le \varepsilon^{-\frac{N}{p}-s}\,|\mathcal{O}|^\frac{1}{p}\,\left(\iint_{\mathcal{O}\times\mathcal{O}}\frac{|u(x)-u(y)|^p}{|x-y|^{N+s\,p}}\,dx\,dy\right)^\frac{1}{p'}\,\left(\int_{\mathcal{O}} |\varphi|^p\,dx\right)^\frac{1}{p}\\
&+C\,\iint_{\{(x,y)\in\mathcal{O}\times(\mathbb{R}^N\setminus\mathcal{O})\, :\, |x-y|\ge \varepsilon\}}\frac{|u(x)|^{p-1}+|u(y)|^{p-1}}{|x-y|^{N+s\,p}}\,|\varphi(x)|\,dx\,dy.
\end{split}
\]
In order to treat the last integral, we observe that
\[
\{(x,y)\in\mathcal{O}\times(\mathbb{R}^N\setminus\mathcal{O})\, :\, |x-y|\ge \varepsilon\}\subset \{(x,y)\in\mathcal{O}\times \mathbb{R}^N\, :\, |x-y|\ge \varepsilon\}.
\]
Thus we obtain
\[
\begin{split}
\iint_{\{(x,y)\in\mathcal{O}\times(\mathbb{R}^N\setminus\mathcal{O})\, :\, |x-y|\ge \varepsilon\}}&\frac{|u(x)|^{p-1}+|u(y)|^{p-1}}{|x-y|^{N+s\,p}}\,|\varphi(x)|\,dx\,dy\\
&\le \int_{\mathcal{O}}\left(\int_{\mathbb{R}^N\setminus B_\varepsilon(x)} \frac{|u(x)|^{p-1}\,|\varphi(x)|}{|x-y|^{N+s\,p}}\,dy\right)\,dx\\
&+\int_{\mathcal{O}}\left(\int_{\mathbb{R}^N\setminus B_\varepsilon(x)} \frac{|u(y)|^{p-1}\,|\varphi(x)|}{|x-y|^{N+s\,p}}\,dy\right)\,dx\\
&\le \frac{N\,\omega_N}{s\,p}\,\varepsilon^{-s\,p}\,\left(\int_{\mathcal{O}} |u|^{p}\,dx\right)^\frac{1}{p'}\,\left(\int_\mathcal{O}|\varphi|^p\,dx\right)^\frac{1}{p}\\
&+\int_\mathcal{O} \left(\int_{\mathbb{R}^N\setminus B_\varepsilon(x)} \frac{|u(y)|^{p-1}}{|x-y|^{N+s\,p}}\,dy\right)\,|\varphi(x)|\,dx.
\end{split}
\]
We conclude by observing that
\[
\sup_{x\in\mathcal{O}} \int_{\mathbb{R}^N\setminus B_\varepsilon(x)} \frac{|u(y)|^{p-1}}{|x-y|^{N+s\,p}}\,dy<+\infty,
\]
thanks to the fact that $u\in L^{p-1}_{s\,p}(\mathbb{R}^N)$, see Lemma \ref{lm:cagatella}.
\end{proof}
In order to use a Moser--type argument for the proof of Theorem \ref{teo:torino}, we will need the following result to guarantee that a certain test function is admissible.
\begin{lm}
\label{lm:minchiata}
Let $1<p<\infty$ and $0<s<1$. Let $\Omega\subset\mathbb{R}^N$ be an open bounded set. For every $u\in W^{s,p}(\Omega)\cap L^\infty(\Omega)$ with compact support in $\Omega$ and $v\in W^{s,p}_{\rm loc}(\Omega)\cap L^\infty(\Omega)$, we have 
\[
u\,v\in W^{s,p}(\Omega).
\]
\end{lm}
\begin{proof}
We start by observing that with simple manipulations we have
\[
\begin{split}
[u\,v]_{W^{s,p}(\Omega)}^p&\le 2^{p-1}\,\iint_{\Omega\times\Omega} \frac{|u(x)-u(y)|^p\,|v(x)|^p}{|x-y|^{N+s\,p}}\,dx\,dy\\
&+2^{p-1}\,\iint_{\Omega\times\Omega} \frac{|v(x)-v(y)|^p\,|u(y)|^p}{|x-y|^{N+s\,p}}\,dx\,dy\\
&\le 2^{p-1}\,\|v\|_{L^\infty(\Omega)}^p\,[u]^p_{W^{s,p}(\Omega)}+2^{p-1}\,\iint_{\Omega\times\Omega} \frac{|v(x)-v(y)|^p\,|u(y)|^p}{|x-y|^{N+s\,p}}\,dx\,dy.
\end{split}
\]
In order to estimate the last integral, we set $\mathcal{O}=\mathrm{supp}(u)$ and then take $\mathcal{O'}$ such that $\mathcal{O}\Subset\mathcal{O}'\Subset\Omega$. We then obtain
\[
\begin{split}
\iint_{\Omega\times\Omega} \frac{|v(x)-v(y)|^p\,|u(y)|^p}{|x-y|^{N+s\,p}}\,dx\,dy&=\iint_{\mathcal{O'}\times \mathcal{O'}} \frac{|v(x)-v(y)|^p\,|u(y)|^p}{|x-y|^{N+s\,p}}\,dx\,dy\\
&+\iint_{(\Omega\setminus \mathcal{O'})\times \mathcal{O}} \frac{|v(x)-v(y)|^p\,|u(y)|^p}{|x-y|^{N+s\,p}}\,dx\,dy\\
&\le \|u\|_{L^\infty(\Omega)}^p\,[v]^p_{W^{s,p}(\mathcal{O}')}\\
&+\frac{2^{p}\,|\Omega|\,|\mathcal{O}|}{\mathrm{dist}(\mathcal{O},\Omega\setminus\mathcal{O}')^{N+s\,p}}\,\|u\|_{L^\infty(\Omega)}^p\,\|v\|^p_{L^\infty(\Omega)}.
\end{split}
\]
This gives the desired conclusion.
\end{proof}

\subsection{An expedient estimate for convex sets}

The following expedient result is a sort of fractional counterpart of the identity 
\[
|\nabla d_K|=1\qquad \mbox{ almost everywhere in }K.
\]
As explained in the Introduction, in the local case this is an essential ingredient in the proof of the Hardy inequality for convex sets. This will play an important role in our case as well.
\begin{prop}
\label{prop:freddo}
Let $1<p<\infty$ and $0<s<1$. Let $K\subset\mathbb{R}^N$ be an open bounded convex set. Then we have
\[
\int_{\{y\in K\, :\, d_K(y)\le d_K(x)\}} \frac{|d_K(x)-d_K(y)|^p}{|x-y|^{N+s\,p}}\,dy\ge \frac{C_1}{1-s}\,d_K(x)^{p\,(1-s)},\quad \mbox{ for a.\,e. }x\in K,
\]
where $C_1=C_1(N,p)>0$ is the constant
\[
C_1=\frac{1}{p}\,\sup_{0<\sigma<1}\Big[\sigma^p\,\mathcal{H}^{N-1}\Big(\{\omega\in\mathbb{S}^{N-1}\, :\, \langle \omega,\mathbf{e}_1\rangle>\sigma\}\Big)\Big].
\]
\end{prop}
\begin{proof}
We set for simplicity $\delta=d_K(x)$, thus $B_\delta(x)\subset K$ and we have
\[
\begin{split}
\int_{\{y\in K\, :\, d_K(y)\le d_K(x)\}} \frac{|d_K(x)-d_K(y)|^p}{|x-y|^{N+s\,p}}\,dy
&\ge \int_{\{y\in B_\delta(x)\, :\, d_K(y)\le d_K(x)\}}  \frac{|d_K(x)-d_K(y)|^p}{|x-y|^{N+s\,p}}\,dy.
\end{split}
\]
We now take $x'\in\partial K$ such that $|x-x'|=\delta$. For a given $0<\sigma<1$, we consider the portion $\Sigma_\sigma(x)$ of $B_\delta(x)$ defined by
\[
\Sigma_\sigma(x)=\left\{y\in B_{\delta}(x)\, :\,\left \langle \frac{y-x}{|y-x|},\frac{x'-x}{|x'-x|}\right\rangle>\sigma \right\},
\]
see Figure \ref{fig:oring}. \begin{figure}[h]
\includegraphics[scale=.35]{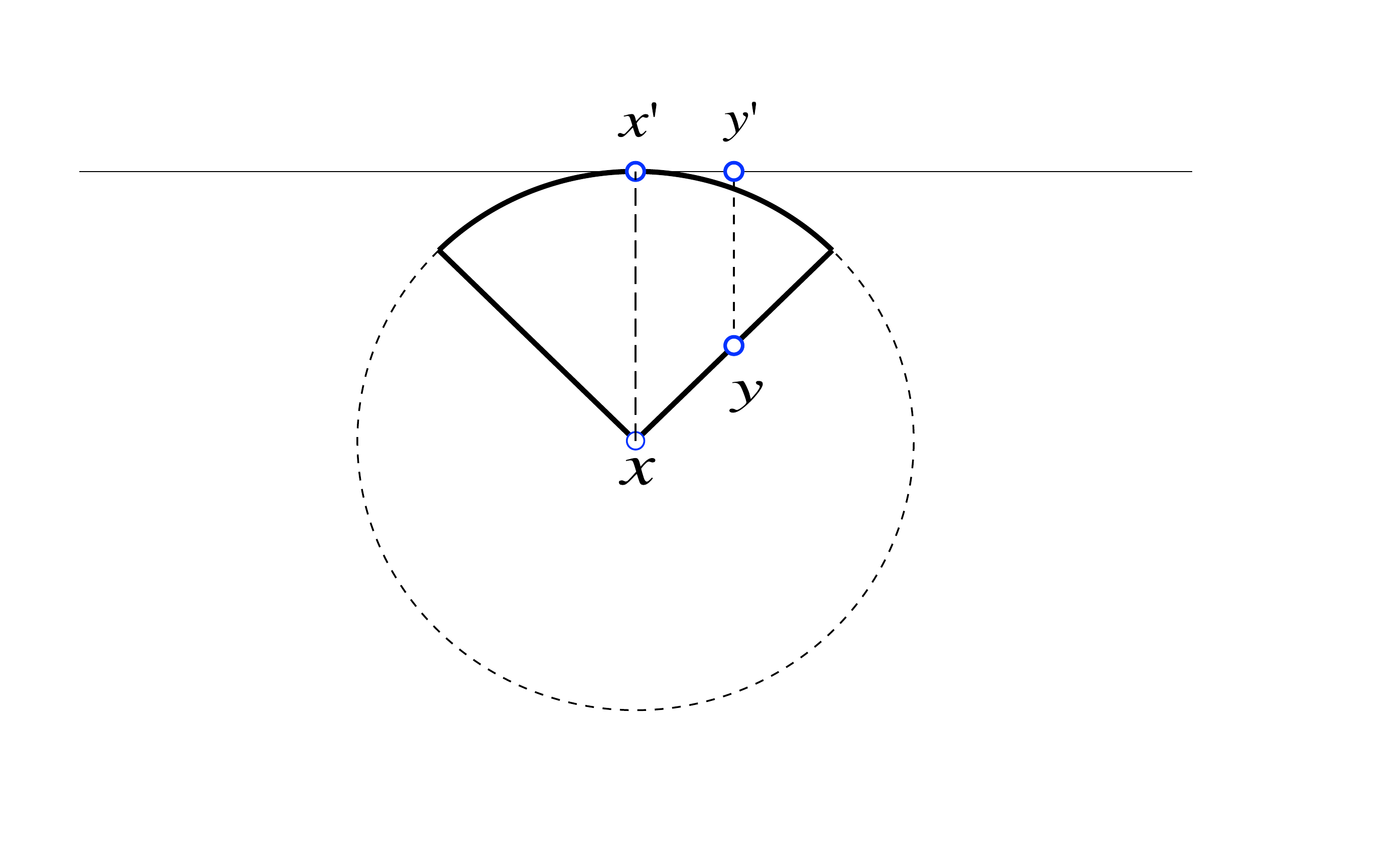}
\caption{The set $\Sigma_\sigma(x)$ and the supporting hyperplane $\Pi_{x'}$.}
\label{fig:oring}
\end{figure}
\par
By convexity of $K$, it is not difficult to see that 
\begin{equation}
\label{inclusione}
\Sigma_\sigma(x)\subset \{y\in B_\delta(x)\, :\, d_K(y)\le d_K(x)\},\qquad \mbox{ for every }0<\sigma<1.
\end{equation}
We can be more precise on this point. We denote by $\Pi_{x'}$ the supporting hyperplane of $K$ at the point $x'$, orthogonal to $x'-x$.
Then for every $y\in K$, we denote by $y'$ the orthogonal projection of $y$ on $\Pi_{x'}$.
Thus by convexity we have
\[
d_K(y)\le |y-y'|,\qquad \mbox{ for every }y\in K.
\]
We then observe that for every $y\in \Sigma_\sigma(x)$, it holds
\begin{equation}
\label{madonnina}
\begin{split}
d_K(x)=|x-x'|&=|y-y'|+\left\langle \frac{y-x}{|y-x|},\frac{x'-x}{|x'-x|}\right\rangle\,|y-x|\\
&\ge d_K(y)+\sigma\,|y-x|.
\end{split}
\end{equation}
By using \eqref{inclusione} and \eqref{madonnina}, we thus obtain
\[
\begin{split}
\int_{\{y\in K\, :\, d_K(y)\le d_K(x)\}} & \frac{|d_K(x)-d_K(y)|^p}{|x-y|^{N+s\,p}}\,dy\\
&\ge \int_{\Sigma_\sigma(x)}  \frac{|d_K(x)-d_K(y)|^p}{|x-y|^{N+s\,p}}\,dy\\
&\ge \sigma^p\,\int_{\Sigma_\sigma(x)} |x-y|^{p\,(1-s)-N}\,dy\\
&=\sigma^p\,\left(\int_{\left\{\omega\in\mathbb{S}^{N-1}\, :\, \left\langle \omega,\frac{x'-x}{|x'-x|}\right\rangle>\sigma\right\}}\,d\mathcal{H}^{N-1}(\omega)\right)\,\int_0^{\delta} \varrho^{-1+p\,(1-s)}\,d\varrho\\
&=\frac{f(\sigma)\,\sigma^p}{(1-s)\,p}\,\delta^{p\,(1-s)},
\end{split}
\]
where $f(\sigma)>0$ is the quantity
\[
f(\sigma)=\int_{\left\{\omega\in\mathbb{S}^{N-1}\, :\, \left\langle \omega,\mathbf{e}_1\right\rangle>\sigma\right\}}\,d\mathcal{H}^{N-1}(\omega).
\]
By arbitrariness of $0<\sigma<1$, we can take the supremum and get the conclusion.
\end{proof}

\section{Superharmonicity of the distance function}
\label{sec:3}

In this section we will prove that $d_K^s$ in a convex set $K$ is  weakly $(s,p)-$superharmonic, see Definition \ref{defi:superarmonica}.
We start with the case of the half-space. The proof of the following fact can be found in \cite[Lemma 3.2]{IMS} (see also \cite[Theorem 3.4.1]{BV} for the case $p=2$).
\begin{lm}
\label{lm:IMS}
We set  $\mathbb{H}^N_+=\{x\in\mathbb{R}^N\, :\, x_N>0\}$.
Let $1<p<\infty$ and $0<s<1$, then $d_{\mathbb{H}^N_+}^s$ is locally weakly $(s,p)-$harmonic in $\mathbb{H}^N_+$.
Moreover, there holds 
\[
\lim_{\varepsilon\to 0} \int_{\mathbb{R}^N\setminus B_\varepsilon(x)} \frac{J_p(d_{\mathbb{H}^N_+}(x)^s-d_{\mathbb{H}^N_+}(y)^s)}{|x-y|^{N+s\,p}}\,dy=0,
\]
strongly in $L^1_{\rm loc}(\mathbb{H}^N_+)$.
\end{lm}
By appealing to the previous result and using the geometric properties of convex sets, we can prove the following
\begin{prop}
\label{prop:guido}
Let $K\subset\mathbb{R}^N$ be an open bounded convex set. For $1<p<\infty$ and $0<s<1$, we have that $d_K^s$ is locally weakly $(s,p)-$superharmonic.
\end{prop}
\begin{proof}
We first observe that $d_K^s$ is locally Lipschitz, bounded and vanishing outside $K$. Thus we have $d_K^s\in W^{s,p}_{\rm loc}(K)\cap L^{p-1}_{s\,p}(\mathbb{R}^N)$. It is sufficient to prove that $d_K^s$ verifies \eqref{super} for every $\varphi\in C^\infty_0(K)$ nonnegative. Thus, let $\varphi\in C^\infty_0(K)$ be nonnegative and call $\mathcal{O}\Subset K$ its support.
We observe that the function
\[
(x,y)\mapsto \frac{J_p(d_K(x)^s-d_K(y)^s)\,\big(\varphi(x)-\varphi(y)\big)}{|x-y|^{N+s\,p}},
\]
is summable. Then by the Dominated Convergence Theorem, we have
\[
\begin{split}
\iint_{\mathbb{R}^N\times\mathbb{R}^N} &\frac{J_p(d_K(x)^s-d_K(y)^s)\,\big(\varphi(x)-\varphi(y)\big)}{|x-y|^{N+s\,p}}\,dx\,dy\\
&=\lim_{\varepsilon\to 0} \iint_{\mathcal{T}_\varepsilon} \frac{J_p(d_K(x)^s-d_K(y)^s)\,\big(\varphi(x)-\varphi(y)\big)}{|x-y|^{N+s\,p}}\,dx\,dy,
\end{split}
\]
where we set $\mathcal{T}_\varepsilon=\{(x,y)\in\mathbb{R}^N\times\mathbb{R}^N\, :\, |x-y|\ge \varepsilon\}$. By Lemma \ref{lm:madonnabellina} we have that for every fixed $\varepsilon$, the function
\[
(x,y)\mapsto \frac{J_p(d_K(x)^s-d_K(y)^s)}{|x-y|^{N+s\,p}}\,\varphi(x),
\]
is summable on $\mathcal{T}_\varepsilon$. Thus we get
\[
\begin{split}
\iint_{\mathcal{T}_\varepsilon} &\frac{J_p(d_K(x)^s-d_K(y)^s)\,\big(\varphi(x)-\varphi(y)\big)}{|x-y|^{N+s\,p}}\,dx\,dy\\
&=\iint_{\mathcal{T}_\varepsilon} \frac{J_p(d_K(x)^s-d_K(y)^s)}{|x-y|^{N+s\,p}}\,\varphi(x)\,dx\,dy-\iint_{\mathcal{T}_\varepsilon} \frac{J_p(d_K(x)^s-d_K(y)^s)}{|x-y|^{N+s\,p}}\,\varphi(y)\,dx\,dy\\
&=2\,\int_{\mathcal{O}} \left(\int_{\mathbb{R}^N\setminus B_\varepsilon(x)}\frac{J_p(d_K(x)^s-d_K(y)^s)}{|x-y|^{N+s\,p}}\,dy\right)\,\varphi(x)\,dx.
\end{split}
\]
In the second equality, we used Fubini's Theorem and the fact that $\varphi$ has support $\mathcal{O}$. In order to conclude, we need to show that 
\begin{equation}
\label{daje}
\lim_{\varepsilon\to 0}\int_{\mathcal{O}} \left(\int_{\mathbb{R}^N\setminus B_\varepsilon(x)}\frac{J_p(d_K(x)^s-d_K(y)^s)}{|x-y|^{N+s\,p}}\,dy\right)\,\varphi(x)\,dx\ge 0.
\end{equation}
We now take $0<\varepsilon\ll 1$ and $x\in \mathcal{O}$,
then we consider a point $x'\in\partial K$ such that $d_K(x)=|x-x'|$. We take a supporting hyperplane to $K$ at $x'$, up to a rigid motion we can suppose that this is given by $\{x\in\mathbb{R}^N\, :\, x_N=0\}$ and that $K\subset \mathbb{H}^N_+$. We observe that by convexity of $K$
\[
d_K(y)\le d_{\mathbb{H}^N_+}(y),\qquad \mbox{ for }y\in \mathbb{R}^N,
\]
and
\[
d_K(x)=|x-x'|=d_{\mathbb{H}^N_+}(x),
\]
see Figure \ref{fig:perotta}.
\begin{figure}[h]
\includegraphics[scale=.3]{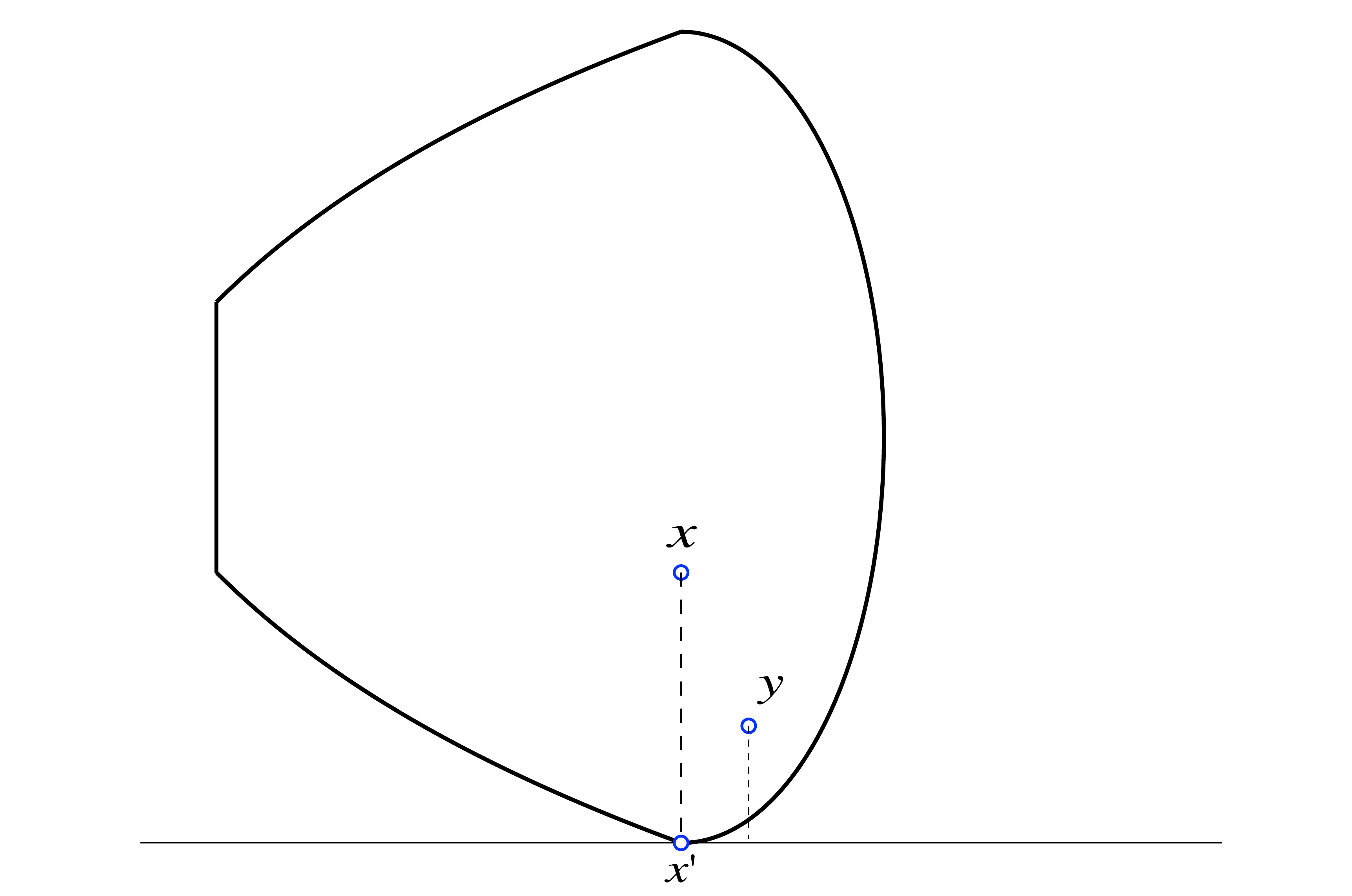}
\caption{The distance of $y$ from $\partial K$ is smaller than its distance from the hyperplane.}
\label{fig:perotta}
\end{figure}
\par
By exploiting these facts and the monotonicity of $J_p$, we obtain for $x\in \mathcal{O}$
\[
\int_{\mathbb{R}^N\setminus B_\varepsilon(x)} \frac{J_p(d_K(x)^s-d_K(y)^s)}{|x-y|^{N+s\,p}}\,dy\ge \int_{\mathbb{R}^N\setminus B_\varepsilon(x)} \frac{J_p(d_{\mathbb{H}^N_+}(x)^s-d_{\mathbb{H}^N_+}(y)^s)}{|x-y|^{N+s\,p}}\,dy.
\]
Observe that the last family of functions converges to $0$ in $L^1(\mathcal{O})$ as $\varepsilon$ goes to $0$, thanks to Lemma \ref{lm:IMS}. 
Thus, by multiplying the previous inequality by $\varphi$ which is nonnegative, integrating over $\mathcal{O}$ and using Lemma \ref{lm:IMS}, we thus obtain
\[
\begin{split}
\lim_{\varepsilon\to 0}\int_{\mathcal{O}} &\left(\int_{\mathbb{R}^N\setminus B_\varepsilon(x)}\frac{J_p(d_K(x)^s-d_K(y)^s)}{|x-y|^{N+s\,p}}\,dy\right)\,\varphi(x)\,dx\\
&\ge \lim_{\varepsilon\to 0}\int_{\mathcal{O}}\left(\int_{\mathbb{R}^N\setminus B_\varepsilon(x)} \frac{J_p(d_{\mathbb{H}^N_+}(x)^s-d_{\mathbb{H}^N_+}(y)^s)}{|x-y|^{N+s\,p}}\,dy\right)\,\varphi(x)\,dx= 0.
\end{split}
\]
This proves \eqref{daje} and thus we get the desired conclusion.
\end{proof}

\section{Proof of Theorem \ref{teo:torino}}
\label{sec:4}

We divide the proof in two parts: we first prove
\begin{equation}
\label{weaker}
\frac{s^p\,A}{1-s}\,\int_K \frac{|u|^p}{d_K^{s\,p}}\,dx\le [u]^p_{W^{s,p}(\mathbb{R}^N)},\qquad \mbox{ for every }u\in C^\infty_0(K),
\end{equation}
with $A=A(N,p)>0$.
Then in the second part we show how to improve the constant in the left-hand side for $s$ close to $0$ and obtain \eqref{nostra}, by using elementary geometric considerations. The proof in this second part is essentially contained in \cite[pages 440--441]{CS}, as pointed out to us by Bart\l omiej Dyda.
\subsection{Proof of inequality \eqref{weaker}}
In turn, we divide the proof in two cases: first we prove the result under the additional assumptions that $K$ is bounded, then we extend it to general convex sets not coinciding with the whole space.
\vskip.2cm\noindent
{\bf Case 1: bounded convex sets.} By Proposition \ref{prop:guido} we know that
\begin{equation}
\label{super1}
\iint_{\mathbb{R}^N\times \mathbb{R}^N}\frac{J_p(d_K(x)^s-d_K(y)^s)(\varphi(x)-\varphi(y))}{|x-y|^{N+s\,p}}\,dx\,dy\ge 0,
\end{equation}
for every nonnegative $\varphi\in W^{s,p}(K)$ with compact support in $K$. Then we test with 
\[
\varphi=\frac{|u|^p}{(d_K^s+\varepsilon)^{p-1}},
\]
where $u\in C^\infty_0(K)$ and $\varepsilon>0$. By Lemma \ref{lm:minchiata}, we have that $\varphi$ is admissible. Indeed, we already know that $d_K^s\in W^{s,p}_{\rm loc}(K)\cap L^\infty(K)$. Moreover, for every $\varepsilon>0$ the function $f(t)=(t+\varepsilon)^{1-p}$ is Lipschitz for $t>0$, thus $(d_K^s+\varepsilon)^{1-p}=f\circ d_K^s\in W^{s,p}_{\rm loc}(K)\cap L^\infty(K)$ as well.
\par
Let us call $\mathcal{O}$ the support of $u$, then from \eqref{super1} we have
\begin{equation}
\label{partiamo?}
\begin{split}
0&\le \iint_{K\times K}\frac{J_p(d_K(x)^s-d_K(y)^s)}{|x-y|^{N+s\,p}}\,\left(\frac{|u(x)|^p}{(d_K(x)^s+\varepsilon)^{p-1}}-\frac{|u(y)|^p}{(d_K(y)^s+\varepsilon)^{p-1}}\right)\,dx\,dy\\
&+2\iint_{\mathcal{O}\times (\mathbb{R}^N\setminus K)}\frac{J_p(d_K(x)^s)}{|x-y|^{N+s\,p}}\,\frac{|u(x)|^p}{(d_K(x)^s+\varepsilon)^{p-1}}\,dx\,dy.
\end{split}
\end{equation}
We first observe that
\begin{equation}
\label{pezzo1}
\iint_{\mathcal{O}\times (\mathbb{R}^N\setminus K)}\frac{J_p(d_K(x)^s)}{|x-y|^{N+s\,p}}\,\frac{|u(x)|^p}{(d_K(x)^s+\varepsilon)^{p-1}}\,dx\,dy\le \iint_{\mathcal{O}\times (\mathbb{R}^N\setminus K)}\frac{|u(x)|^p}{|x-y|^{N+s\,p}}\,dx\,dy.
\end{equation}
We now need to estimate the double integral
\[
\mathcal{I}:=\iint_{K\times K}\frac{J_p(d_K(x)^s-d_K(y)^s)}{|x-y|^{N+s\,p}}\,\left(\frac{|u(x)|^p}{(d_K(x)^s+\varepsilon)^{p-1}}-\frac{|u(y)|^p}{(d_K(y)^s+\varepsilon)^{p-1}}\right)\,dx\,dy.
\]
For this, we crucially exploit the fundamental inequality of Lemma \ref{lm:bacini}, with the choices
\[
a=d_K(x)^s+\varepsilon,\quad b=d_K(y)^s+\varepsilon,\quad c=|u(x)|,\quad d=|u(y)|.
\]
This entails
\begin{equation}
\label{eccosci!}
\begin{split}
\mathcal{I}&\le -C_2\,\iint_{K\times K} \left|\frac{d_K(x)^s-d_K(y)^s}{d_K(x)^s+d_K(y)^s+2\,\varepsilon}\right|^p\,(|u(x)|^p+|u(y)|^p)\,\frac{dx\,dy}{|x-y|^{N+s\,p}}\\
&+C_3\,\iint_{K\times K} \frac{\big||u(x)|-|u(y)|\big|^p}{|x-y|^{N+s\,p}}\,dx\,dy,
\end{split}
\end{equation}
where $C_2$ and $C_3$ are as in Lemma \ref{lm:bacini}.
By using \eqref{eccosci!} in \eqref{partiamo?}, together with \eqref{pezzo1}, we obtain
\begin{equation}
\label{guasi}
C_2\,\iint_{K\times K} \left|\frac{d_K(x)^s-d_K(y)^s}{d_K(x)^s+d_K(y)^s}\right|^p\,(|u(x)|^p+|u(y)|^p)\,\frac{dx\,dy}{|x-y|^{N+s\,p}}\le C_3\,\big[|u|\big]^p_{W^{s,p}(\mathbb{R}^N)}.
\end{equation}
To obtain \eqref{guasi}, we also took the limit as $\varepsilon$ goes to $0$ and used Fatou's Lemma.
We observe that by symmetry, we have 
\begin{equation}
\label{hidden}
\begin{split}
\iint_{K\times K} &\left|\frac{d_K(x)^s-d_K(y)^s}{d_K(x)^s+d_K(y)^s}\right|^p\,(|u(x)|^p+|u(y)|^p)\,\frac{dx\,dy}{|x-y|^{N+s\,p}}\\
&=2\,\iint_{K\times K} \left|\frac{d_K(x)^s-d_K(y)^s}{d_K(x)^s+d_K(y)^s}\right|^p\,|u(x)|^p\,\frac{dx\,dy}{|x-y|^{N+s\,p}}\\
&\ge 2\,\int_K\left(\int_{\{y\in K\, :\, d_K(y)\le d_K(x)\}}\left|\frac{d_K(x)^s-d_K(y)^s}{d_K(x)^s+d_K(y)^s}\right|^p\,\frac{dy}{|x-y|^{N+s\,p}}\right)\,|u(x)|^p\,dx.
\end{split}
\end{equation}
We now use the pointwise inequality \eqref{gufetti}, so to obtain
\[
\begin{split}
\iint_{K\times K} &\left|\frac{d_K(x)^s-d_K(y)^s}{d_K(x)^s+d_K(y)^s}\right|^p\,(|u(x)|^p+|u(y)|^p)\,\frac{dx\,dy}{|x-y|^{N+s\,p}}\\
&\ge \frac{s^p}{2^{p-1}}\,\int_K \left(\int_{\{y\in K\, :\, d_K(y)\le d_K(x)\}}\frac{|d_K(x)-d_K(y)|^p}{|x-y|^{N+s\,p}}\,dy\right)\,\frac{|u(x)|^p}{d_K(x)^{p}}\,dx.
\end{split}
\]
By using this in \eqref{guasi} and then applying the expedient estimate of Proposition \ref{prop:freddo}, we end up with
\[
\frac{s^p\,A}{1-s}\,\int_K \frac{|u|^p}{d_K^{s\,p}}\,dx\le [u]^p_{W^{s,p}(\mathbb{R}^N)},
\]
where we have used the triangle inequality to replace the seminorm of $|u|$ with that of $u$, i.e.
\[
\Big||u(x)|-|u(y)|\Big|\le |u(x)-u(y)|.
\]
This concludes the proof of \eqref{weaker}. We observe that 
\begin{equation}
\label{costante1}
A=\frac{C_1}{2^{p-1}}\frac{C_2}{C_3},
\end{equation}
where $C_1=C_1(N,p)$ is the constant of Proposition \ref{prop:freddo}, and $C_2,\,C_3$ (which depend only on $p$) come from Lemma \ref{lm:bacini}.
\vskip.2cm\noindent
{\bf Case 2: general convex sets.} We now take $K\not=\mathbb{R}^N$ an open unbounded convex set. For every $R>0$ we set $K_R=K\cap B_R(0)$. Let us take $u\in C^\infty_0(K)$, then for every $R$ large enough, we have $u\in C^\infty_0(K_R)$ as well. By using the previous case, we then get
\[
\frac{s^p\,A}{1-s}\,\int_{K} \frac{|u|^p}{d_{K_R}^{s\,p}}\,dx=\frac{s^p\,A}{1-s}\,\int_{K_R} \frac{|u|^p}{d_{K_R}^{s\,p}}\,dx\le [u]_{W^{s,p}(\mathbb{R}^N)}^p.
\]  
By observing that $d_{K_R}\le d_K$, we then get the desired conclusion.

\subsection{Improved constant for $s$ close to $0$}
\label{sec:rotto}

For every $x\in K$,  we take $x_0\in \mathbb{R}^N\setminus K$ such that 
\[
|x-x_0|=2\,d_K(x)\qquad \mbox{ and }\qquad \frac{x+x_0}{2}\in\partial K.
\] 
Then we can estimate
\[
\begin{split}
\iint_{\mathbb{R}^N\times\mathbb{R}^N} \frac{|u(x)-u(y)|^p}{|x-y|^{N+s\,p}}\,dx\,dy\ge \int_K |u(x)|^p\,\left(\int_{(\mathbb{R}^N\setminus K)\setminus B_{d_K(x)}(x_0)} \frac{dy}{|x-y|^{N+s\,p}}\right)\,dx.
\end{split}
\]
We observe that for every $y\in \mathbb{R}^N\setminus B_{d_K(x)}(x_0)$, we have
\[
|x-y|\le|x-x_0|+|x_0-y|=2\,d_K(x)+|x_0-y|\le 3\,|x_0-y|. 
\]
By convexity, we have that (see Figure \ref{fig:sverzina})
\[
K_x:=\{y\in \mathbb{R}^N\setminus B_{d_K(x)}\, :\, \langle y-x_0,x-x_0\rangle <0\}\subset (\mathbb{R}^N\setminus K)\setminus B_{d_K(x)}(x_0).
\]
\begin{figure}
\includegraphics[scale=.35]{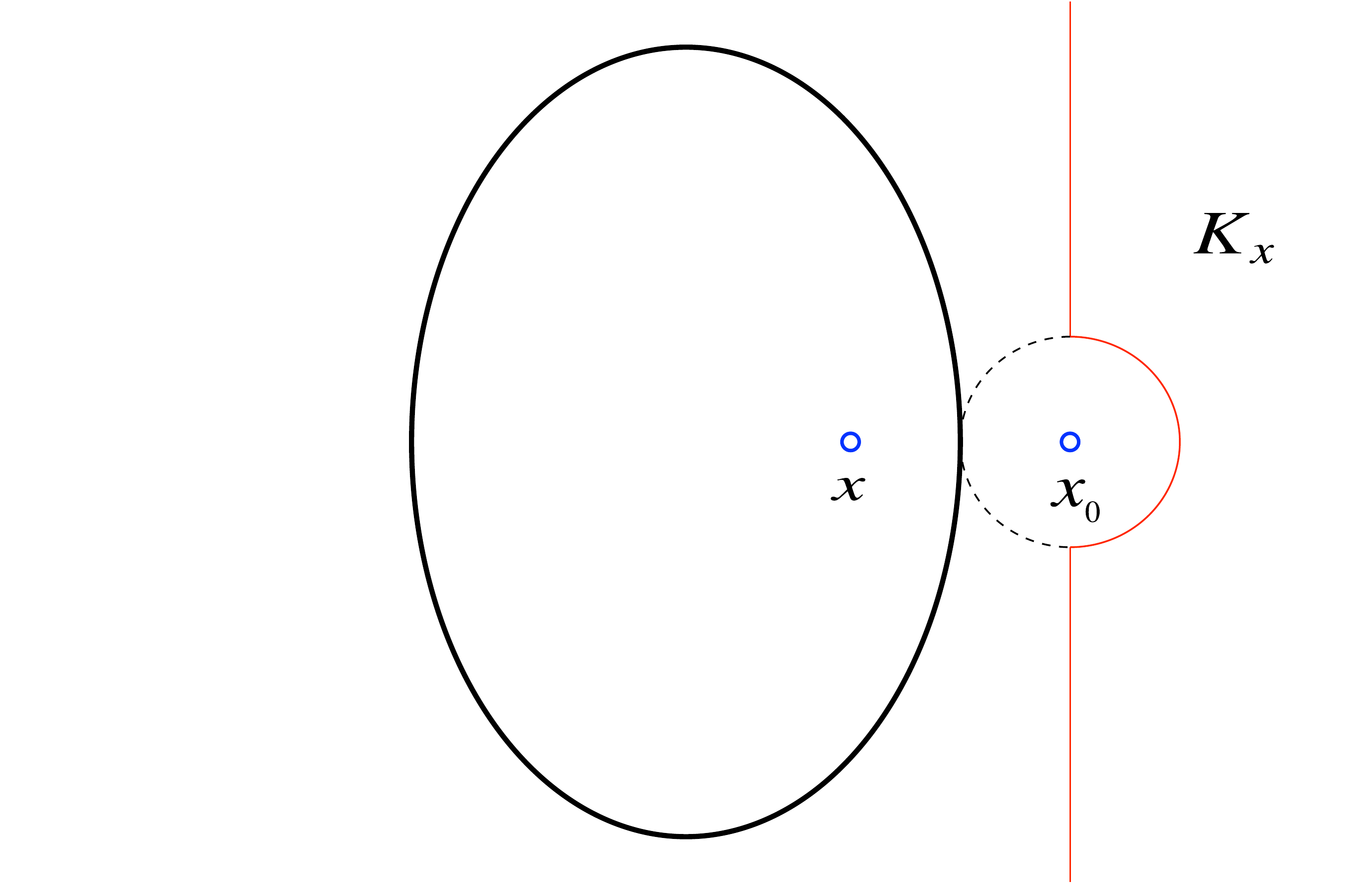}
\caption{The set $K_x$ in the second part of the proof of Theorem \ref{teo:torino}.}
\label{fig:sverzina}
\end{figure}
By joining the last two informations, we get
\[
\begin{split}
\iint_{\mathbb{R}^N\times\mathbb{R}^N} \frac{|u(x)-u(y)|^p}{|x-y|^{N+s\,p}}\,dx\,dy&\ge \left(\frac{1}{3}\right)^{N+s\,p}\,\int_K |u(x)|^p\,\left(\int_{K_x} \frac{dy}{|x_0-y|^{N+s\,p}}\right)\,dx\\
&=\left(\frac{1}{3}\right)^{N+s\,p}\,\left(\int_{\{\omega\in\mathbb{S}^{N-1}\, :\, \langle\omega,\mathbf{e}_1\rangle>0\}}d\mathcal{H}^{N-1} \right)\\
&\times \int_K \left(\int_{d_K(x)}^{+\infty} \varrho^{-1-s\,p}\,d\varrho\right)\,|u(x)|^p\,dx\\
&=\frac{N\,\omega_N}{2}\,\frac{1}{s\,p}\,\left(\frac{1}{3}\right)^{N+s\,p}\,\int_K\frac{|u|^p}{d_K^{s\,p}}\,dx.
\end{split}
\]
In conclusion, we have obtained for every $u\in C^\infty_0(K)$
\begin{equation}
\label{weakerB}
\frac{B}{s}\,\int_K \frac{|u|^p}{d_K^{s\,p}}\,dx\le [u]_{W^{s,p}(\mathbb{R}^N)}^p,\qquad \mbox{ with } B=\frac{1}{p}\,\frac{N\,\omega_N}{2}\,3^{-N-p}.
\end{equation}
By joining \eqref{weaker} and \eqref{weakerB}, we finally get
\[
\frac{A\,s^{p+1}+(1-s)\,B}{2\,s\,(1-s)}\,\int_K \frac{|u|^p}{d_K^{s\,p}}\,dx\le [u]_{W^{s,p}(\mathbb{R}^N)}^p,\qquad \mbox{ for every }u\in C^\infty_0(K).
\]
By observing that the quantity $A\,s^{p+1}+(1-s)\,B$ is bounded from below by a positive constant independent of $s$, we finally get the desired inequality \eqref{nostra}.

\begin{oss}\label{oss:constant}
Let us set $\Phi(s)=A\,s^{p+1}+(1-s)\,B$, where $A$ and $B$ are defined in \eqref{costante1} and \eqref{weakerB}. It is easy to see that the constant $\mathcal C$ appearing in \eqref{nostra} is given by
\[
\mathcal C=\min_{s\in(0,1)} \frac{\Phi(s)}{2}=\frac{1}{2}\,\left\{\begin{array}{rl}
B\left(1-\dfrac{p}{p+1}\,\left(\dfrac{B}{(p+1)A}\right)^{1/p}\right),& \mbox{ if } B<(p+1)\,A,\\
&\\
A, & \mbox{ if } B\ge (p+1)\,A.
\end{array}
\right.
\]
\end{oss}

\section{Some consequences}
\label{sec:consequence} For an open set $\Omega\subset\mathbb{R}^N$, we define the {\it homogeneous Sobolev-Slobodecki\u{\i} space} $\mathcal{D}^{s,p}_0(\Omega)$ as the completion of $C^\infty_0(\Omega)$ with respect to the norm
\[
u\mapsto [u]_{W^{s,p}(\mathbb{R}^N)}.
\]
We also define the {\it first eigenvalue of the fractional $p-$Laplacian of order $s$ in $\Omega$}, i.e.
\[
\lambda^s_{1,p}(\Omega):=\inf_{u\in C^\infty_0(\Omega)}\left\{[u]^p_{W^{s,p}(\mathbb{R}^N)}\, :\, \int_\Omega |u|^p\,dx=1\right\}.
\]
This is the sharp constant in the fractional Poincar\'e inequality
\[
C\,\int_\Omega |u|^p\,dx\le [u]^p_{W^{s,p}(\mathbb{R}^N)},\qquad \mbox{ for every }u\in C^\infty_0(\Omega).
\]
We observe that $\lambda^s_{1,p}(\Omega)>0$ is equivalent to the continuity of the embedding $\mathcal{D}^{s,p}_0(\Omega)\hookrightarrow L^p(\Omega)$.
\par
We highlight a couple of consequences of our main result, in terms of lower bounds on $\lambda^s_{1,p}$. As usual, we pay particular attention to the factor $s\,(1-s)$.
\begin{coro}
\label{coro:poincare}
Let $1<p<\infty$ and $0<s<1$. Let $K\subset\mathbb{R}^N$ be an open convex set such that 
\[
R_K:=\sup_{x\in K} d_K(x)<+\infty.
\]
Then $\mathcal{D}^{s,p}_0(K)$ is a functional space, continuously embedded in $L^p(K)$. Moreover, it holds
\[
\frac{\mathcal{C}}{R_K^{s\,p}}\le s\,(1-s)\,\lambda^s_{1,p}(K),
\]
where $\mathcal{C}$ is the same constant as in Theorem \ref{teo:torino}.
\end{coro}
\begin{proof}
This is a straightforward consequence of \eqref{nostra} and of the definition of $R_K$.
\end{proof}
The quantity $R_K$ above is called {\it inradius of } $K$. Observe that this is the radius of the largest ball inscribed in $K$.
\par
For a general open set, we have the following
\begin{coro}[Poincar\'e inequality for sets bounded in one direction]
Let $\omega_0\in\mathbb{R}^N$ be such that $|\omega_0|=1$ and let $\ell_1,\ell_2\in\mathbb{R}$ with $\ell_1<\ell_2$.  For every open set such that 
\[
\Omega\subset \{x\in\mathbb{R}^N\, :\, \ell_1<\langle x,\omega_0\rangle<\ell_2\}.
\]
we have
\[
\mathcal{C}\,\left(\frac{2}{\ell_2-\ell_1}\right)^{s\,p}\le s\,(1-s)\,\lambda^s_{1,p}(\Omega),
\]
where $\mathcal{C}$ is the same constant as in Theorem \ref{teo:torino}.
\end{coro}
\begin{proof}
We set for simplicity $\mathcal{S}=\{x\in\mathbb{R}^N\, :\, \ell_1<\langle x,\omega_0\rangle<\ell_2\}$. Then by domain inclusion we directly obtain $\lambda^s_{1,p}(\Omega)\ge \lambda^s_{1,p}(\mathcal{S})$. It is now sufficient to use Corollary \ref{coro:poincare} for the convex set $\mathcal{S}$, for which $R_{\mathcal{S}}=(\ell_2-\ell_1)/2$.
\end{proof}
\appendix 

\section{Some pointwise inequalities}

We collect here some pointwise inequalities needed throughout the whole paper. The most important one is Lemma \ref{lm:bacini}. We recall the notation
\[
J_p(t)=|t|^{p-2}\,t,\qquad \mbox{ for }t\in\mathbb{R}.
\]
\begin{lm}
\label{lm:sco}
Let $1<p<\infty$, for every $a,b> 0$ we have
\[
J_p(a-b)\,\left(\frac{1}{b^{p-1}}-\frac{1}{a^{p-1}}\right)\ge (p-1)\,|\log b-\log a|^p.
\]
Equality holds if and only if $a=b$.
\end{lm}
\begin{proof}
This is proved in \cite[Lemma A.2 \& Remark A.3]{BP}.
\end{proof}
\begin{lm}
\label{lm:acazzo}
For every $a,b>0$ we have
\[
\frac{|a-b|}{a+b}\le |\log a-\log b|.
\]
Equality holds if and only if $a=b$.
\end{lm}
\begin{proof}
We observe that if $a=b$ there is nothing to prove. We then take $a\not=b$ and without loss of generality we can suppose $a>b$. The seeked inequality is then equivalent to
\[
\frac{a-b}{a+b}\le \log \frac{a}{b},\qquad \mbox{ for } 0<b<a.
\]
By setting $t=b/a$, this in turn is equivalent to prove that
\[
\frac{1-t}{1+t}\le -\log t,\qquad \mbox{ for }0<t<1.
\]
By basic Calculus, it is easily seen that the function
\[
\varphi(t)=\log t+\frac{1-t}{1+t},
\]
is strictly increasing for $t\in (0,1)$ and $\varphi(1)=0$. This gives the desired conclusion.
\end{proof}
\begin{oss}
By combining Lemma \ref{lm:sco} and \ref{lm:acazzo}, we also obtain
\begin{equation}
\label{lucio}
J_p(a-b)\,\left(\frac{1}{b^{p-1}}-\frac{1}{a^{p-1}}\right)\ge (p-1)\,\left|\frac{a-b}{a+b}\right|^p,
\end{equation}
for every $a,b>0$ and $1<p<\infty$.
\end{oss}
\begin{lm}
Let $0<s<1$, then for every $a,b>0$ we have
\begin{equation}
\label{gufetti}
\frac{|a^s-b^s|}{a^s+b^s}\ge \frac{s}{2}\,\frac{|a-b|}{\max\{a,\,b\}}.
\end{equation}
\end{lm}
\begin{proof}
For $a=b$ there is nothing to prove. Without loss of generality, we can assume $a>b$. By defining $t=b/a\in(0,1)$, inequality \eqref{gufetti} is equivalent to prove
\[
\frac{1-t^s}{1+t^s}\ge \frac{s}{2}\,(1-t).
\]
We observe that by the ``below tangent property'' of concave functions, we have
\[
t^s\le 1+s\,(t-1)\qquad \mbox{ i.\,e. }\quad 1-t^s\ge s\,(1-t).
\]
By combining this with the trivial estimate $1+t^s\le 2$, we get the conclusion.
\end{proof}
An essential ingredient in the proof of our main result has been the following pointwise inequality. 
\begin{lm}[Fundamental inequality]
\label{lm:bacini}
Let $1<p<\infty$ and let $a,b,c,d\in\mathbb{R}$, with $a,b>0$ and $c,d\ge 0$. Then there exist two constants $C_2=C_2(p)>0$ and $C_3=C_3(p)>1$, such that
\begin{equation}
\label{lodep}
J_p(a-b)\,\left(\frac{c^p}{a^{p-1}}-\frac{d^p}{b^{p-1}}\right)+C_2\, \left|\frac{a-b}{a+b}\right|^p\,(c^p+d^p)\le C_3\,|c-d|^p.
\end{equation}
\end{lm}
\begin{proof}
We observe that for $a=b$ there is nothing to prove, since the left-hand side vanishes. Without loss of generality, we can assume $a>b$.
Also notice that if $c\le d$, then
\[
\begin{split}
(a-b)^{p-1}\,\left(\frac{c^p}{a^{p-1}}-\frac{d^p}{b^{p-1}}\right)&\le (a-b)^{p-1}\,\left(\frac{d^p}{a^{p-1}}-\frac{d^p}{b^{p-1}}\right)\\
&=- d^p\,(a-b)^{p-1}\,\left(\frac{1}{b^{p-1}}-\frac{1}{a^{p-1}}\right)\\
&\le -(p-1)\,d^p\,\left|\frac{a-b}{a+b}\right|^p\le -(p-1)\,\frac{c^p+d^p}{2}\,\left|\frac{a-b}{a+b}\right|^p,
\end{split}
\]
where in the second inequality we used \eqref{lucio}. Thus inequality \eqref{lodep} holds with $C_2=(p-1)/2$ and $C_3>0$ arbitrary.
\vskip.2cm\noindent
We assume now that $a>b$ and $c>d$, then by setting 
\[
t=b/a\in (0,1) \qquad \mbox{ and }\qquad A=d/c\in[0,1),
\] 
inequality \eqref{lodep} is equivalent to
\begin{equation}
\label{torp}
(1-t)^{p-1}\,\left(1-\frac{A^p}{t^{p-1}}\right)+C_2\,\left(\frac{1-t}{1+t}\right)^p\,(1+A^p)\le C_3\,(1-A)^p,
\end{equation}
with $t\in(0,1)$ and $A\in[0,1)$.
We study the function
\begin{equation}\label{phi}
\Phi(t)=(1-t)^{p-1}\,\left(1-\frac{A^p}{t^{p-1}}\right),\qquad t\in(0,1),
\end{equation}
which is maximal for $t=A$. This in particular implies\footnote{We observe that this is equivalent to 
\[
J_p(a-b)\,\left(\frac{c^p}{a^{p-1}}-\frac{d^p}{b^{p-1}}\right)\le |c-d|^p,
\]
which is a discrete version of {\it Picone's inequality}, see \cite[Proposition 4.2]{BF}.}
\begin{equation}
\label{picone}
(1-t)^{p-1}\,\left(1-\frac{A^p}{t^{p-1}}\right)\le (1-A)^p.
\end{equation}
We now distinguish two cases: 
\[
\mbox{ either }\qquad 0\le A\le \frac{1}{2}\qquad \mbox{ or }\qquad \frac{1}{2}<A<1.
\]
$\boxed{\mbox{\bf A. Case $0\le A\le 1/2$.}}$ This is the simplest case. Indeed, we have
\[
\frac{1-t}{1+t}\le 1\qquad \mbox{ and }\qquad (1-A)^p\ge \frac{1}{2^p}.
\]
Thus by using this and \eqref{picone}, we get
\[
(1-t)^{p-1}\,\left(1-\frac{A^p}{t^{p-1}}\right)+\frac{C_3-1}{2^{p+1}}\,\left(\frac{1-t}{1+t}\right)^p\,(1+A^p)\le C_3\,(1-A)^p,
\]
which is \eqref{torp} with $C_2=(C_3-1)/2^{p+1}$ and $C_3>1$ arbitrary.
\vskip.2cm\noindent
$\boxed{\mbox{\bf B. Case $1/2< A<1$.}}$ Here in turn we consider two subcases: $A\le t$ and $0<t<A$.
\vskip.2cm\noindent
{\tt B.1. Case $1/2< A<1$ and $t\ge A$.} This is easy, since we directly have
\[
(1-t)^p\le (1-A)^p,
\]
and thus
\[
\left(\frac{1-t}{1+t}\right)^p\le (1-t)^p\le (1-A)^p.
\]
By using this and \eqref{picone}, we get
\[
(1-t)^{p-1}\,\left(1-\frac{A^p}{t^{p-1}}\right)+\frac{C_3-1}{2}\,\left(\frac{1-t}{1+t}\right)^p\,(1+A^p)\le C_3\,(1-A)^p,
\]
which is \eqref{lodep} with $C_2=(C_3-1)/2$ and $C_3> 1$ arbitrary.
\vskip.2cm\noindent
{\tt B.1. Case $1/2< A<1$ and $0<t<A$.} Here we need to study in more details the function $\Phi$ defined in \eqref{phi}.
We have
\[
\begin{split}
\Phi''(t)&=(p-1)\,(1-t)^{p-3}\,\left[p-2+2\,\frac{A^p}{t^p}-p\,\frac{A^p}{t^{p+1}}\right]\\
&=(p-1)\,(1-t)^{p-2}\,\left[\frac{p-2}{1-t}\,\left(1-\frac{A^p}{t^p}\right)-p\,\frac{A^p}{t^{p+1}}\right],\qquad t\in (0,1).
\end{split}
\]
By an easy computation, we can see that
\[
t\mapsto p-2+2\,\frac{A^p}{t^p}-p\,\frac{A^p}{t^{p+1}}
\]
is monotone increasing, thus we get
\[
\Phi''(t)\le (p-1)\,(1-t)^{p-3}\,p\,\left(1-\frac{1}{A}\right),\qquad \mbox{ for }0<t<A.
\]
In particular, we get that $\Phi$ is concave on the interval $(0,A)$.
We use a second order Taylor expansion around the maximum point $t=A$, i.e.
\begin{equation}
\label{sviluppino}
\begin{split}
\Phi(t)&=\Phi(A)+\int_t^A \Phi''(s)\,(s-t)\,ds\\
&=(1-A)^p\\
&+(p-1)\,\int_t^A (1-s)^{p-2}\,\left[\frac{p-2}{1-s}\,\left(1-\frac{A^p}{s^p}\right)-p\,\frac{A^p}{s^{p+1}}\right]\,(s-t)\,ds,
\end{split}
\end{equation}
where we used that $\Phi'(A)=0$.
In order to estimate the remainder term inside the integral, we distinguish once again two cases:
\begin{itemize}
\item if $1<p\le 2$, then by using Lemma \ref{lm:rutto} below and the fact that $A\ge 1/2$, we get
\[
\begin{split}
\int_t^A &(1-s)^{p-2}\,\left[\frac{p-2}{1-s}\,\left(1-\frac{A^p}{s^p}\right)-p\,\frac{A^p}{s^{p+1}}\right]\,(s-t)\,ds\\
&=\int_t^A \left[(p-2)\,\frac{s^p-A^p}{1-s}-p\,\frac{A^p}{s}\right]\,\frac{(1-s)^{p-2}}{s^p}\,(s-t)\,ds\\
&\le -\frac{p\,(p-1)}{2^p}\,\int_t^A  \frac{(1-s)^{p-2}}{s^p}\,(s-t)\,ds\\
&\le -\frac{p\,(p-1)}{2^p}\,(1-t)^{p-2}\,\int_t^A (s-t)\,ds\\
&=-\frac{p\,(p-1)}{2^{p+1}}\,(1-t)^{p-2}\,(A-t)^2.
\end{split}
\] 
By using the previous estimate in \eqref{sviluppino}, we have
\begin{equation}
\label{p1}
(1-t)^{p-1}\,\left(1-\frac{A^p}{t^{p-1}}\right)\le (1-A)^p-\frac{p\,(p-1)^2}{2^{p+1}}\,(1-t)^{p-2}\,(A-t)^2.
\end{equation}
It is now sufficient to observe that
\[
\begin{split}
(1-t)^p&=(1-t)^{p-2}\,(1-t)^2\\
&\le (1-t)^{p-2}\,\Big(2\,(A-t)^2+2\,(1-A)^2\Big)\\
&\le 2\,(1-t)^{p-2}\,(A-t)^2+2\,(1-A)^p
\end{split}
\]
and thus
\begin{equation}
\label{p2}
C_2\,\left(\frac{1-t}{1+t}\right)^p\,(1+A^p)\le 4\,C_2\,(1-t)^{p-2}\,(A-t)^2+4\,C_2\,(1-A)^p.
\end{equation}
If we sum up \eqref{p1} and \eqref{p2} and choose $C_3>1$ arbitrary and
\[
C_2=\frac{1}{4}\,\min\left\{C_3-1,\,\frac{p\,(p-1)^2}{2^{p+1}}\right\},
\] 
we get again the desired conclusion \eqref{torp}.
\vskip.2cm
\item if $p> 2$, then by using that $s\le A$
\[
\begin{split}
\int_t^A &(1-s)^{p-2}\,\left[\frac{p-2}{1-s}\,\left(1-\frac{A^p}{s^p}\right)-p\,\frac{A^p}{s^{p+1}}\right]\,(s-t)\,ds\\
&\le -p\,A^p\,\int_t^A  \frac{(1-s)^{p-2}}{s^{p+1}}\,(s-t)\,ds\\
&\le -\frac{p}{2^p}\,\int_t^A (1-s)^{p-2}\,(s-t)\,ds.
\end{split}
\]
The last integral can be explicitly computed, integrating by parts: we have
\[
\begin{split}
\int_t^A (1-s)^{p-2}\,(s-t)\,ds&=-\frac{1}{p-1}\,(1-A)^{p-1}\,(A-t)+\frac{1}{p\,(p-1)}\,(1-t)^p\\
&-\frac{1}{p\,(p-1)}\,(1-A)^p.
\end{split}
\]
We use Young's inequality to estimate the first term on the right-hand side
\[
\begin{split}
-\frac{1}{p-1}\,(1-A)^{p-1}\,(A-t)&\ge -\frac{1}{p}\,\varepsilon^{-\frac{1}{p-1}}\,(1-A)^p\\
&-\frac{\varepsilon}{p\,(p-1)}\,(A-t)^p,
\end{split}
\]
with $\varepsilon>0$.
We use these estimates in \eqref{sviluppino}. This in turn gives
\[
\begin{split}
(1-t)^{p-1}\,\left(1-\frac{A^p}{t^{p-1}}\right)&\le (1-A)^p+\frac{p-1}{2^p}\,\varepsilon^{-\frac{1}{p-1}}\,(1-A)^p\\
&+\frac{\varepsilon}{2^p}\,(A-t)^p\\
&-\frac{1}{2^p}\,(1-t)^p+\frac{1}{2^p}\,(1-A)^p.
\end{split}
\]
By choosing $\varepsilon=1/2$ and using that $A-t\le 1-t$, we then obtain
\[
\begin{split}
(1-t)^{p-1}\,\left(1-\frac{A^p}{t^{p-1}}\right)&\le C_3\,(1-A)^p-\frac{1}{2^{p+1}}\,(1-t)^p,
\end{split}
\]
with
\[
C_3=1+\frac{p-1}{2^p}\,2^\frac{1}{p-1}+\frac{1}{2^p}.
\]
Once again, this is enough to get the desired conclusion, since 
\[
C_2\,\left(\frac{1-t}{1+t}\right)^p\,(1+A^p)\le 2\,C_2\,(1-t)^{p}.
\]
Then we only need to choose $C_2=1/2^{p+2}$ in order to get \eqref{torp}.
\end{itemize}
We thus concluded the proof.
\end{proof}
\begin{oss}[The constants $C_2$ and $C_3$]
An inspection of the proof reveals that in the case $1<p\le 2$, the constant $C_3$ can be chosen arbitrarily close to $1$. Accordingly, we have 
\[
\begin{split}
C_2&=\min\left\{\frac{C_3-1}{4},\,\frac{1}{4}\,\frac{p\,(p-1)^2}{2^{p+1}},\,\frac{C_3-1}{2^{p+1}},\frac{p-1}{2}\right\}\\
&=\frac{1}{2^{p+1}}\,\min\left\{\frac{p\,(p-1)^2}{4},\,C_3-1\right\},
\end{split}
\]
and thus it degenerates to $0$ as $C_3 \searrow 1$.
\end{oss}
\begin{oss}
\label{oss:spiegone}
Inequality \eqref{lodep} looks similar to the pointwise inequality which can be found right before \cite[equation (3.12), page 1289]{DKP}, where the term
\[
\left|\frac{a-b}{a+b}\right|^p\,(c^p+d^p),
\]
is replaced by
\[
\left|\log a-\log b\right|^p\,\min\{c^p,\,d^p\}.
\]
The main difference is that in our inequality the terms $c$ and $d$ play a symmetric role. This means that the quantity $(c^p+d^p)$ is unchanged when we exchange the roles of $c$ and $d$, while this is not the case for $\min\{c^p,\, d^p\}$.
This property is a crucial feature in order to prove Theorem \ref{teo:torino}: precisely, this is hidden in the estimate \eqref{hidden}.
 On the other hand, it is easy to see that the inequality in \cite{DKP} {\it can not} have this property, i.e. one can not replace  
\[
\left|\log a-\log b\right|^p\,\min\{c^p,\,d^p\},
\]
by 
\[
\left|\log a-\log b\right|^p\,(c^p+d^p).
\]
Thus the inequality of \cite{DKP} does not seem useful in order to prove Hardy inequality.
\end{oss}
In the previous result, we needed the following inequality in order to deal with the case $1<p\le 2$.
\begin{lm}
\label{lm:rutto}
Let $1<p\le 2$, for every $s\in(0,1)$ and $A\in[0,1]$ we have
\[
(p-2)\,\frac{s^p-A^p}{1-s}-p\,\frac{A^p}{s}\le -p\,(p-1)\,A^p.
\]
\end{lm}
\begin{proof}
We rewrite
\begin{equation}
\label{rutto}
(p-2)\,\frac{s^p-A^p}{1-s}-p\,\frac{A^p}{s}=(2-p)\,\frac{A^p-s^p}{1-s}-p\,\frac{A^p}{s}.
\end{equation}
We then observe that 
\begin{equation}
\label{ele}
\frac{A^p-s^p}{1-s}\le p\,A^p.
\end{equation}
Indeed, the latter is equivalent to
\[
A^p\,(1-p+s\,p)\le s^p,
\]
which is easily seen to be true. 
By using \eqref{ele} in \eqref{rutto}, we now obtain
\[
\begin{split}
(p-2)\,\frac{s^p-A^p}{1-s}-p\,\frac{A^p}{s}&\le (2-p)\,p\,A^p-p\,\frac{A^p}{s}\le -p\,(p-1)\,A^p,
\end{split}
\]
as desired.
\end{proof}

\end{document}